\documentclass{amsart}

\usepackage[bookmarks=true]{hyperref}       

\usepackage[top=3cm,bottom=3cm,left=3.5cm,right=3.5cm]{geometry}

\usepackage{amssymb,latexsym,amsmath,amscd}
\usepackage[T1]{fontenc}
\usepackage{xspace}
\usepackage{enumerate}
\usepackage{graphicx}
\usepackage{todonotes}

\theoremstyle{plain}
\newtheorem{theorem}{Theorem}[section]
\newtheorem*{theorem*}{Theorem}

\newtheorem{lemma}[theorem]{Lemma}

\theoremstyle{definition}
\newtheorem{definition}[theorem]{Definition}

\newtheorem{remark}[theorem]{Remark}

\newcommand{\enm}[1]{\ensuremath{#1}}          %

\newcommand{\cal}[1]{\mathcal{#1}}

\newcommand{\NN}{\enm{\mathbb{N}}}

\renewcommand{\AA}{\enm{\mathbb{A}}}
\newcommand{\PP}{\enm{\mathbb{P}}}

\newcommand{\Ii}{\enm{\cal{I}}}

\newcommand{\Oo}{\enm{\cal{O}}}

\newcommand{\calA}{\mathcal{A}}

\renewcommand{\phi}{\varphi}
\renewcommand{\theta}{\vartheta}
\renewcommand{\epsilon}{\varepsilon}

\title[A note on cactus rank]{A note on the cactus rank for Segre-Veronese varieties}

\author[E. Ballico]{Edoardo Ballico}
\author[A. Bernardi]{Alessandra Bernardi}
\address[Edoardo Ballico, Alessandra Bernardi]{Dipartimento di Matematica,  Univ. Trento, Italy}
\email[E. Ballico]{edoardo.ballico@unitn.it}
\email[A. Bernardi]{alessandra.bernardi@unitn.it}

\author[F. Gesmundo]{Fulvio Gesmundo}
\address[Fulvio Gesmundo]{QMATH, Dept. of Mathematical Sciences, Univ. of Copenhagen, Denmark}
\email{fulges@math.ku.dk}


\newcommand{\bbP}{\mathbb{P}}
\newcommand{\bbK}{\mathbb{K}}

\newcommand{\calO}{\mathcal{O}}

\newcommand{\vvirg}{, \dots ,}
\newcommand{\bfd}{\mathbf{d}}

\newcommand{\charact}{\mathrm{char}}

\begin{document}
\begin{abstract}We give an upper bound for the cactus rank of any multi-homogeneous polynomial. \end{abstract}

\maketitle
\section{Introduction}

The notion of cactus rank for homogeneous polynomials was first introduced in the literature by A. Iarrobino and V. Kanev in their famous book \cite[Definition 5.1, Definition 5.66]{ik} with the name of ``scheme length''. The expression ``cactus rank'' was first used  in \cite{br,rs} after the pioneering work of W. Buczy\'nska and J. Buczy\'nski \cite{bb} where the authors introduced the notion of cactus varieties.

Let $\bbK$ be an algebraically closed field and let $S = \bbK_{DP} [x_0 \vvirg x_n]$ be the divided power ring of polynomials in $n+1$ variables over $\bbK$ (if $\charact{(\bbK)} = 0$ one can work with the standard polynomial ring -- we refer to \cite[Appendix A]{ik}, \cite[Section 2.1]{bjmr}, \cite[Appendix 2]{Eis:CommutativeAlgebra} for the theory on divided power rings). Denote by $S_d$ the component of degree $d$ in $S$ and let $\nu_d$ denote the $d$-th Veronese embedding $\nu_d : \mathbb{P}^n \rightarrow \bbP ( S_d) = \mathbb{P}^{\binom{n+d}{d}-1}$ given by $[L] \mapsto [L^{[d]}]$ (for $f \in S$ homogeneous, $[f]$ denotes its class in the corresponding projective space), namely the embedding of $\bbP ^n$ defined by the line bundle $\calO(d)$; its image $\nu_d ( \bbP^n) :=  \{ [L^{[d]}] :  L \in S_1 \} \subseteq \bbP S_d $ is the so-called Veronese variety. Given a homogeneous polynomial $f \in S_d$ (here we assume $d > \charact(\bbK)$ if $\bbK$ has positive characteristic), the \emph{cactus rank} of $f$  is the minimum length of a zero-dimensional scheme $Z$ contained in $\nu_d (\bbP^n)$ such that such that $[f]\in \langle Z \rangle$.

Equivalently, the cactus rank of $f$ is the minimum length of a zero-dimensional scheme $\Gamma \subseteq \bbP^n$ that is \emph{apolar} to $f$, in the sense of \cite[Def. 1.11]{ik}. In this setting $Z = \nu_d(\Gamma)$. The scheme $\Gamma$ is called a minimal apolar scheme to $f$ and it has been extensively studied in the recent years (see \cite{ik,bb,br,rs,bjmr,bbm,bj}). In particular in \cite{br,bjmr} the authors explore the behaviors for the cactus rank in the case of cubics polynomials, while in \cite{rs} the monomial case is studied. 

The notion of cactus rank can be defined with respect to an arbitrary algebraic variety (see e.g. \cite{g2}). In this work, we focus on the setting of partially symmetric tensors, which has been studied also in \cite{BERNARDI20081542,grv,t,hoos}, where the authors used a natural extension of the notion of apolarity, and in \cite{g1,g2}, where the author widened  this notion to an even more general setting, and provided equations for cactus varieties via vector bundle techniques (see e.g. \cite{lo,t}).

Given positive integers $s, n_1 \vvirg n_s$, let $R = \bbK _{DP} [x_{i, j}: j = 1 \vvirg s, i= 0 \vvirg n_j]$ be a polynomial ring in $\sum (n_j + 1)$ variables. Given a multidegree $\bfd = (d_1 \vvirg d_s)$, write $R_\bfd$ for the multihomogeneous component of multidegree $\bfd$, i.e. $f \in R_\bfd$ if, for every $j$, $f$ is homogeneous of degree $d = d_1 + \cdots + d_s$ in $R$ and in addition it is homogeneous of degree $d_j$ in the $j$-th set of variables $\{ x_{i,j} : i = 0 \vvirg n_j\}$. 

Let $Y=\mathbb{P}^{n_1}\times \cdots \times \mathbb{P}^{n_s}$, where we may regard $\bbP^{n_j} = \bbP (R_{e_j})$ with $e_j = (0 \vvirg 0,1,0 \vvirg 0)$ ($1$ being at the $j$-th entry); the Segre-Veronese embedding of multidegree $\bfd = (d_1 \vvirg d_s)$ is the map $\nu _{\bfd}: Y \rightarrow \bbP R_\bfd  = \PP^r$ where $r+1 = \prod _{i=1}^{s} \binom{n_i+d_i}{n_i}$, namely the embedding of $Y$ defined by the line bundle $\mathcal{O}_Y(\bfd)$ (notice $R_\bfd = H^0(\mathcal{O}_Y(\bfd))$). The notion of cactus rank extends to this setting as follows.

\begin{definition}\label{cactus:def}
Let $X:= \nu _{\bfd}(Y)\subset \PP^r$ be the Segre-Veronese variety, $r+1 = \prod _{i=1}^{s} \binom{n_i+d_i}{n_i}$ (here $\sum d_i < \charact(\bbK)$ if we work in positive characteristic). For any $f\in \PP^r$ the cactus rank $cr_\bfd(f)$ is the minimal length of a zero-dimensional scheme $Z\subset X$ such that $f\in \langle Z\rangle$.
\end{definition}

In this work, we give an upper bound on the maximum possible value of $cr_{\bfd}$ in terms of $d_1 \vvirg d_s$ and $n_1 \vvirg n_s$, or equivalently on the value of $cr_{\bfd}$ for every $f \in R_\bfd$ (Theorem \ref{i1}).

Possible approaches to this problem include the use of the notion of Hankel operators as it is applied in \cite{bbcm} to partially symmetric tensors, or of extended notions of apolarity as in \cite{hoos,grv,g1}. Our approach is based on the one followed by M. V. Catalisano, A. V. Geramita and A. Gimigliano in \cite{cgg}, where the authors show that the Hilbert function of a certain number of double fat points on the Segre-Veronese variety coincides with the Hilbert function of a scheme contained in $\mathbb{P}^n$ consisting of fat points and linear spaces, for $n = n_1 + \cdots + n_s$.  This will allow us to give an explicit upper bound in the case of Segre-Veronese varieties.

Following \cite{cgg} we construct an embedding 
\[
u: R_\bfd = H^0(\Oo _Y(d_1,\dots ,d_s))\to H^0(\PP^n,\Oo _{\PP^n}(d))
\]
with $n:= n_1+\cdots +n_s$ and $d:= d_1+\cdots +d_s$, which will allow us to extend the results on the cactus rank of homogeneous polynomials to the case of multi-homogeneous polynomials.

Our main result is the following:

\begin{theorem}\label{i1}
For any multi-homogeneous polynomial $f\in \PP^r\simeq \mathbb{P} R_{\bfd}$ with $r+1 = \prod _{i=1}^{s} \binom{n_i+d_i}{n_i}$, $n:= n_1+\cdots +n_s$ and $d:= d_1+\cdots +d_s$, we have that 
\begin{itemize}
\item if $d$ is even, say $d=2k+2$ for some $k\in \NN$, then $cr_{\bfd}(f) \le \binom{n+k}{k}+\binom{n+k+1}{k+1}$,
\item if $d =2k+1$ is odd we have $cr_{\bfd}(f) \le 2\binom{n+k}{k}$.
\end{itemize}
\end{theorem}

\begin{remark}
We compare our result with some known results on this topic. For $s = 1$, we obtain an upper bound for the cactus rank in the symmetric case: for $d = 3$, Thm. \ref{i1} recovers the result of \cite{br} on cubic forms; for higher values of $d$, Thm. \ref{i1} does not give an improvement with respect to the bound provided by \cite{rs}.

In the tensor case (when $d_i = 1$ for $i=1 \vvirg s$), we obtain more interesting results: if $s$ is odd and $m = n_i$ for every $i$, Thm. \ref{i1} provides the upper bound $2 \binom{m + \lfloor s/2 \rfloor}{\lfloor s/2 \rfloor}$ which grows as a polynomial of degree $\lfloor s/2 \rfloor$ in $ms$; a similar growth is attained in the case where $s$ is even. In particular, in the case $s = 3$, we obtain that the maximum value of $cr_{1,1,1}$ grows with leading term $6m$ which was already known to Buczy\'nski (personal communication, 2016); we observe that a bound with the same leading term can be obtained as a straightforward application of the result of \cite{br}. For every fixed value of $s$, Thm. \ref{i1} shows that the maximum possible cactus rank grows as a polynomial of degree $s/2$ in $m$ (a lower bound with this growth is immediate via flattening techniques), whereas the maximum possible value for the tensor rank (see e.g. \cite{Lan:TensorBook}) grows as a polynomial of degree $s-1$ in $m$.

We point out that during the referee process of this paper we have been informed that M. Ga\l{}\k{a}zka is currently working on the same topic and he independently obtained bounds similar to the ones of Thm. \ref{i1} using different techniques.
\end{remark}

\subsection*{Acknowledgements} This paper was conceived during the International workshop on ``Quantum Physics and Geometry'' (Levico Terme, Trento, Italy, July 4--6, 2017), funded by CIRM (Centro Internazionale per la Ricerca Matematica), INDAM (Istituto Nazionale di Alta Matematica), Universit\`a degli Studi di Trento, INO-CNR BEC Center, SiQuro project of Provincia Autonoma di Trento, EU-FET Proactive grant AQuS, INFN (Istituto Nazionale di Fisica Nucleare), TIFPA (Trento Institute for Fundamental Physics and Applications). F.G. acknowledges financial support from the European Research Council (ERC Grant Agreement no. 337603), the Danish Council for Independent Research (Sapere Aude), and VILLUM FONDEN via the QMATH Centre of Excellence (Grant no. 10059).

We thank J.M. Landsberg for suggesting to investigate upper bounds on the cactus rank and J. Buczy\'nski for pointing out the uncited references in an earlier version of this paper. We thank the anonymous referee for useful remarks.

\section{Map construction}

Consider the set of variables $\{ z_0 \} \cup \{ z_{i,j}\}_{1\le j\le s, 1\le i \le n_j}$ and regard them, as homogeneous coordinates on $\PP^n$, with $n = n_1 + \cdots n_s$; let $S:=\mathbb{K}_{\mathrm{DP}}[z_0,z_{i,j}]_{1\le j\le s, 1\le i \le n_j}$ be the corresponding polynomial ring and let $S_d$ be its component of degree $d$. For any fixed $j=1,\dots ,s$, let $M_j\subset \PP^n$ be the linear subspace with equations $z_0=0$ and $z_{i,h} =0$ {for all $i$ and for $h\in \{1,\dots ,s\}\setminus \{j\}$}. Note that the only non-zero variables on $M_j$ are $z_{1,j},\dots ,z_{n_j,j}$ and so $\dim M_j=n_j-1$. For every $j$, let $L_j :=(d-d_j)M_j\subset \PP^n$ be the scheme defined by the $(d-d_j)$-th power  $(I(M_j))^{(d-d_j)}$ of the ideal of $M_j$. The equations of $L_j$ are all monomials of degree $d-d_j$ in the variables $\{z_0,z_{i,h}\}_{h\in \{1,\dots ,s\}\setminus \{j\}}$.

Let
$$\mathcal{A}:= H^0(\PP^n,\Ii _{L_1\cup \cdots \cup L_s}(d)).$$ 

Following \cite{cgg}, we construct an injective linear map $u: R_{\bfd}\to H^0(\PP^n,\Oo _{\PP^n}(d))=S_d$ whose image is $\mathcal{A}$. The map $u$ is the evaluation map defined by 
\begin{equation}\label{eqa1}
\begin{array}{lll}
 u(x_{0,j}) = z_0 &\quad \text{ for } j = 1 \vvirg s \\
 u(x_{i,j}) = z_{i,j}& \quad  \text{ for } j = 1 \vvirg s, i = 1 \vvirg n_s.   
\end{array}
\end{equation}

In particular, $u:  R_{\bfd} \to S_{d_1+ \cdots + d_s}$ is obviously linear on the graded components of interest.

\section{Proof of the Main Theorem}

\begin{lemma}\label{lemma1}
The evaluation $u : R_\bfd \to S_d$ is injective.  
\end{lemma}

\begin{proof}
 Take $f\in R_{\bfd}$ and write 
 \[
f =\sum _{0\le a_1\le d_1,\dots ,0\le a_s \le d_s} x_{0,1}^{a_1}\cdots x_{0,s}^{a_s} F_{a_1,\dots ,a_s},
\]
with $F_{a_1,\dots ,a_s}\in R_{d_1-a_1,\dots ,d_s-a_s}$ and no $x_{0,j}$ appearing in $F_{a_1,\dots ,a_s}$, so that, for every $j$, $F_{a_1 , \dots, a_s}$ is homogeneous of degree $d_j - a_j$ in $\{ x_{i,j} \}_{i = 1, \dots , n_j}$. Therefore
 \[
u(f) = \sum _{0\le a_1\le d_1,\dots ,0\le a_s \le d_s} z_0^{a_1 + \cdots +a_s} F_{a_1,\dots ,a_s} ( z_{ij} ),
 \]
 where $F_{a_1,\dots ,a_s} ( z_{i,j} )$ is just the image of the evaluation $ x_{i,j} \mapsto z_{i,j}$. Now, if $u(f) = 0$, then for every $\rho \in \{ 0 , \dots , d\}$, we have $\sum_{a_1 + \cdots + a_s = \rho} F_{a_1,\dots ,a_s}(z_{i,j}) = 0$; this is a sum of linearly independent terms, because the summands have different multidegree in the variables $\{z_{i,j}\}$. This shows $F_{a_1,\dots ,a_s}(z_{i,j}) = 0$ for every choice of $(a_1 , \dots , a_s)$ and therefore $f = 0$.
\end{proof}

\begin{lemma}\label{lemma2}
The image of $u: R_\bfd \to S_d$ is $\mathcal{A}$. 
\end{lemma}

\begin{proof} Recall $\mathcal{A} = H^0(\bbP^n, \mathcal{I}_{L_1 \cup \cdots \cup L_s}(d))$. We first prove that $u(R_\bfd) \subseteq \mathcal{A}$. Fix $f\in R_{\bfd}$ and $j\in \{1,\dots ,s\}$. It is sufficient to prove that $u(f)\in H^0(\PP^n,\Ii _{L_j}(d))$ for every $j$. Since $f$ has degree $d_j$ with respect to the variables $\{x_{0,j},x_{1,j},\dots ,x_{n_j,j}\}$, every monomial appearing with non-zero
coefficient in $f$ has degree at most $d_j$ with respect to $\{x_{1,j},\dots ,x_{n_j,j}\}$; therefore each monomial appearing with non-zero coefficient in $u(f)$ has degree
at most $d_j$ with respect to $\{z_{1,1},\dots ,z_{n_j,j}\}$ and hence degree at least $d-d_j$ with respect to the other homogeneous coordinates of $\PP^n$. This is equivalent to saying that $u(f)\in H^0(\PP^n,\Ii _{L_j}(d))$ by the definitions of $M_j$ and of $L_j$. The inclusion $u(H^0(\Oo _Y(\bfd))) \supseteq \mathcal{A}$, now follows from the fact that $u(H^0(\Oo _Y(\bfd)))$ and $ \mathcal{A}$ have the same dimension, which is a consequence of the injectivity of $u$ (Lemma \ref{lemma1}), and of the fact that $R_{\bfd}$ and $\mathcal{A}$ have the same dimension (case $Z=W=\emptyset$ of \cite[Theorem 1.1]{cgg}).
\end{proof}

\begin{remark}
The map $u$ induces an isomorphism between the open set $\AA^n \subseteq \bbP^n$ with $z_0 \neq 0$ and the open set $U \subseteq Y$ with $x_{0,j} \neq 0$. Explicitly, this isomorphism is 
\begin{align*}
 U &\to \AA^n \\
((1,\xi_{1,1} \vvirg \xi_{n_1,1}) \vvirg (1,\xi_{1,s}\vvirg \xi_{n_s,s})) &\mapsto ( \xi_{i,j} : i = 1 \vvirg n_j, j = 1 \vvirg s),
\end{align*}
where $(\xi_{1,1} \vvirg \xi_{n_j,1})$ are coordinates on the open set $\AA^{n_j} \subseteq \bbP^{n_j}$ of the $j$-th factor of $Y$. In particular, any zero-dimensional scheme $W\subset \AA^n=U\subset \PP^n$ corresponds isomorphically to a zero-dimensional scheme $Z\subset Y$ of the same degree.
\end{remark}

\begin{proof}[Proof of Theorem \ref{i1}.] 
Let $\nu _d: \PP^n \to \PP^N$, $N= \binom{n+d}{n}-1$, be the order $d$ Veronese embedding. The point $f$ is (the class of) a multi-homogeneous polynomial; let $ p = u(f) \in \mathbb{P}^N$ be its image via $u$, so that, from Lemma \ref{lemma2}, $p \in \mathbb{P}\calA$. Following \cite{br}, let $N_d = \binom{n+k}{k}+\binom{n+k+1}{k+1}$ if $d=2k+2$ is even) and $N_d =  2\binom{n+k}{k}$ if $d=2k+1$ is odd. Theorem 3 in \cite{br} shows that $cr(p) \leq N_d$ and in fact that there is a punctual (i.e. connected) zero-dimensional scheme $W \subset \AA^n \subseteq \PP^n$ such that $\deg(W) \leq N_d$ and $p\in \langle \nu _d(W)\rangle$. If $\bbK$ has positive characteristic, (in this case we assume $\charact(\bbK) > d$ as in Definition \ref{cactus:def}) we follow \cite[Section 2.1 and Section 4]{bjmr}. Let $Z = u^{-1} (W) \subset \AA^n\subset Y$ be the corresponding zero-dimensional scheme in $Y$ obtained via the isomorphism induced by $u$. We have $p = u(f) \in \langle \nu_d( W) \rangle$ and therefore $f \in \langle \nu_{\bfd} (Z) \rangle$.
\end{proof}

\bibliographystyle{plain}
\bibliography{bibQcactus}

\begin{thebibliography}{10}

\bibitem{BERNARDI20081542}
A.~Bernardi.
\newblock Ideals of varieties parameterized by certain symmetric tensors.
\newblock {\em Journal of Pure and Applied Algebra}, 212(6):1542 -- 1559, 2008.

\bibitem{bbcm}
A.~Bernardi, J.~Brachat, P.~Comon, and B.~Mourrain.
\newblock {General tensor decomposition, moment matrices and applications}.
\newblock {\em J. Symb. Comp.}, 52:51--71, 2013.

\bibitem{bbm}
A.~Bernardi, J.~Brachat, and B.~Mourrain.
\newblock {A comparison of different notions of ranks of symmetric tensors}.
\newblock {\em Lin. Alg. Appl.}, 460:205--230, 2014.

\bibitem{bjmr}
A.~Bernardi, J.~Jelisiejew, P.M. Marques, and K.~Ranestad.
\newblock {On polynomials with given {H}ilbert function and applications}.
\newblock {\em Collectanea Mathematica}, pages 39--64, 2018.

\bibitem{br}
A.~Bernardi and K.~Ranestad.
\newblock {On the cactus rank of cubic forms}.
\newblock {\em J. Symb. Comp.}, 50:291--297, 2013.

\bibitem{bb}
W.~Buczy\'{n}ska and J.~Buczy\'{n}ski.
\newblock {Secant varieties to high degree {V}eronese reembeddings,
  catalecticant matrices and smoothable {G}orenstein schemes}.
\newblock {\em J. Algebraic Geom.}, 23(1):63--90, 2014.

\bibitem{bj}
J.~Buczy\'{n}ski and J.~Jelisiejew.
\newblock {Finite schemes and secant varieties over arbitrary characteristic}.
\newblock {\em Diff. Geom. Appl.}, 55:13--67, 2017.

\bibitem{cgg}
M.~V. Catalisano, A.~V. Geramita, and A.~Gimigliano.
\newblock {Higher secant varieties of {S}egre-{V}eronese varieties}.
\newblock In {\em {Projective varieties with unexpected properties}}, pages
  81--107. Walter de Gruyter GmbH \& Co. KG, Berlin, 2005.

\bibitem{Eis:CommutativeAlgebra}
D.~Eisenbud.
\newblock {\em {Commutative {A}lgebra: with a view toward algebraic geometry}},
  volume 150 of {\em {Graduate Texts in Mathematics}}.
\newblock Springer-Verlag, New York, 1995.

\bibitem{g1}
M.~Ga\l{}\k{a}zka.
\newblock {Multigraded apolarity}.
\newblock {\em arXiv:1601.06211}, 2016.

\bibitem{g2}
M.~Ga\l{}\k{a}zka.
\newblock {Vector bundles give equations of cactus varieties}.
\newblock {\em Lin. Alg. Appl.}, 521:254--262, 2017.

\bibitem{grv}
M.~Gallet, K.~Ranestad, and N.~Villamizar.
\newblock {Varieties of apolar subschemes of toric surfaces}.
\newblock {\em Arkiv f{\"o}r Matematik}, 56(1):73--99, 2016.

\bibitem{hoos}
J.~D. Hauenstein, L.~Oeding, G.~Ottaviani, and A.~J. Sommese.
\newblock {Homotopy techniques for tensor decomposition and perfect
  identifiability}.
\newblock {\em J. Reine Angew. Math}, 2016.

\bibitem{ik}
A.~Iarrobino and V.~Kanev.
\newblock {\em {Power sums, {G}orenstein algebras, and determinantal loci}},
  volume 1721 of {\em {Lecture Notes in Mathematics}}.
\newblock Springer-Verlag, Berlin, 1999.
\newblock Appendix C by Iarrobino and S. L. Kleiman.

\bibitem{Lan:TensorBook}
J.~M. Landsberg.
\newblock {\em {Tensors: {G}eometry and {A}pplications}}, volume 128 of {\em
  {Graduate Studies in Mathematics}}.
\newblock American Mathematical Society, Providence, RI, 2012.

\bibitem{lo}
J.~M. Landsberg and G.~Ottaviani.
\newblock {Equations for secant varieties of {V}eronese and other varieties}.
\newblock {\em Ann. Mat. Pura Appl. (4)}, 192(4):569--606, 2013.

\bibitem{rs}
K.~Ranestad and F.-O. Schreyer.
\newblock {On the rank of a symmetric form}.
\newblock {\em J. Algebra}, 346:340--342, 2011.

\bibitem{t}
Z.~Teitler.
\newblock Geometric lower bounds for generalized ranks.
\newblock {\em arXiv:1406.5145}, 2014.

\end{thebibliography}

\end{document}